\documentclass[11pt]{amsart}
\usepackage[left=2.2cm,tmargin=2.5cm,bmargin=2.5cm,right=2.2cm]{geometry}

\usepackage{amsfonts}
\usepackage{amssymb}
\usepackage{hyperref}
\usepackage{stmaryrd}
\usepackage{mathrsfs}
\usepackage[capitalise]{cleveref}
\usepackage{amscd}
\usepackage{physics}

\usepackage{tipa}

\usepackage{tikz}
\usetikzlibrary{matrix,arrows,decorations.pathmorphing,cd}

\theoremstyle{plain}
\newtheorem{thm}{Theorem}[section]

\newtheorem{cor}[thm]{Corollary}
\newtheorem{prop}[thm]{Proposition}

\theoremstyle{definition}

\theoremstyle{remark}

\newcommand{\lemref}[1]{\hyperref[#1]{Lemma \ref*{#1}}}
\newcommand{\thmref}[1]{\hyperref[#1]{Theorem \ref*{#1}}}
\newcommand{\propref}[1]{\hyperref[#1]{Proposition \ref*{#1}}}
\newcommand{\corref}[1]{\hyperref[#1]{Corollary \ref*{#1}}}
\newcommand{\defref}[1]{\hyperref[#1]{Definition \ref*{#1}}}
\newcommand{\remref}[1]{\hyperref[#1]{Remark \ref*{#1}}}
\newcommand{\conjref}[1]{\hyperref[#1]{Conjecture \ref*{#1}}}

\makeatletter
\newcommand*{\defeq}{\mathrel{\rlap{%
                     \raisebox{0.27ex}{$\m@th\cdot$}}%
                     \raisebox{-0.27ex}{$\m@th\cdot$}}%
                     =}

\numberwithin{equation}{section}
\makeatother

\makeatletter
\def\@setcopyright{}
\def\serieslogo@{}
\makeatother

\input xy
\xyoption{all} 

\title{Hecke's Theorem on the Different for $3$-Manifolds}

\author{Will Sawin}\thanks{W.S. served as a Clay Research Fellow while working on this paper.}

\author{Mark Shusterman}

\address{Department of Mathematics, Columbia University, New York, NY 10027, USA}
\email{sawin@math.columbia.edu}

\address{Department of Mathematics, Harvard University, 1 Oxford Street, Cambridge, MA 02138, USA}
\email{mshusterman@math.harvard.edu}

\begin{document}

\begin{abstract}

Hecke has shown that the different of an extension of number fields is a square in the class group.
We prove an analog for branched covers of closed $3$-manifolds saying that the branch divisor is a square in the first homology group.

\end{abstract}

\maketitle

\section{Introduction}

Let $E/F$ be an extension of number fields, let $\mathcal O_E$ be the ring of integers of $E$, and let $\operatorname{Cl}(\mathcal O_E)$ be the class group of $\mathcal O_E$.
One associates to the extension $E/F$ the different $\mathcal D_{E/F}$, an ideal in $\mathcal O_E$, see \cite[Chapter 3]{Serre}. 
Hecke has shown that as an element of $\operatorname{Cl}(\mathcal O_E)$, the different $\mathcal D_{E/F}$ is a square, namely there exists an ideal class $J \in \operatorname{Cl}(\mathcal O_E)$ such that $J^2 = \mathcal D_{E/F}$ in $\operatorname{Cl}(\mathcal O_E)$.
Hecke's proof uses a reciprocity formula for Gauss sums, see \cite{Arm} and \cite{Fro} for a proof and a discussion of related results.

An analog of Hecke's theorem for finite separable extensions of fields of fractions of Dedekind domains fails in general, see \cite{FST}. 
However, there exists an analog in case $E/F$ is a finite separable extension of function fields of curves over finite fields of odd characteristic, see \cite{Arm}.
Another geometric analog of Hecke's theorem, based on similarities between the inverse of the different and the canonical bundle on a curve, is the theory of theta characteristics.

In this work we consider an analog of Hecke's theorem for $3$-manifolds, as suggested by arithmetic topology. We refer to \cite{Mori} for the analogy between rings of integers and primes on the one hand, and $3$-manifolds and knots on the other hand.
The analog of $\operatorname{Spec}(\mathcal O_F)$ is a closed (not necessarily oriented) $3$-manifold $M$.
The map $\operatorname{Spec}(\mathcal O_E) \to \operatorname{Spec}(\mathcal O_F)$ is replaced by a cover $\pi \colon \widetilde{M} \to M$ branched over a link $L \subset M$,
so $\widetilde{M}$ is a closed $3$-manifold and $\pi^{-1}(M \setminus L)$ is a covering space of $M \setminus L$. 
The inverse image of $L$ under $\pi$ is a link $\widetilde{L}$ in $\widetilde{M}$. 

For a prime ideal $\mathfrak{p}$ of $\mathcal O_E$ we denote by $e_{\mathfrak{p}}$ its ramification index, namely the largest positive integer $e$ for which $\mathfrak p^e$ contains $\mathfrak p \cap \mathcal O_F$. We view $\operatorname{Spec}(\mathcal O_E) \to \operatorname{Spec}(\mathcal O_F)$ as branched over the primes of $\mathcal O_E$ that ramify, so $\widetilde L$ is our analog for $\mathcal R_{E/F} = \{\mathfrak p \in \operatorname{Spec}(\mathcal O_E) : e_{\mathfrak p} > 1 \}$.
The analogy is perhaps closest in case $\operatorname{Spec}(\mathcal O_E) \to \operatorname{Spec}(\mathcal O_F)$ is tamely ramified, namely $e_{\mathfrak p}$ is coprime to $|\mathcal O_E/\mathfrak p|$ for every $\mathfrak p \in \operatorname{Spec}(\mathcal O_E)$.
In this case the different of $E/F$ is given by
\[
\mathcal D_{E/F} = \prod_{\mathfrak p \in \mathcal R_{E/F}} \mathfrak p^{e_{\mathfrak p} - 1}.
\]

The prime ideals in $\mathcal R_{E/F}$ are analogous to the components of the link $\widetilde L$.
For each component $\widetilde{K}$ of this link, let the ramification index $e_{\widetilde{K}}$ be the number of times the image under $\pi$ of a small loop around $\widetilde{K}$ wraps around $\pi(\widetilde K)$. An analog of $\operatorname{Cl}(\mathcal O_E)$ is $H_1(\widetilde M, \mathbb Z)$, and a homology class is a square if and only if its image in
\[
H_1(\widetilde M, \mathbb Z) \otimes_{\mathbb Z} \mathbb Z/ 2 \mathbb Z \cong H_1(\widetilde M, \mathbb Z/2 \mathbb Z)
\]
vanishes.
Our analogy of $\mathcal D_{E/F}$, or rather of its class in $\operatorname{Cl}(\mathcal O_E)/\operatorname{Cl}(\mathcal O_E)^2$, is the branch divisor 
\[
\mathcal D_\pi = \sum_{ \widetilde{K} \ \text{a component of} \ \widetilde{L}} (e_{\widetilde{K}}-1)  [\widetilde{K}] \in H_1(\widetilde M, \mathbb Z/2 \mathbb Z)
\]
of $\pi$.
Since we are working with $\mathbb Z / 2\mathbb Z$-coefficients, it is not necessary to fix an orientation of $\widetilde{K}$, nor is the sign of $e_{\widetilde{K}}$ significant.

\begin{thm} \label{MainRes}

Let $\widetilde M$ and $M$ be closed $3$-manifolds, and let $\pi \colon \widetilde M \to M$ be a cover branched over a link in $M$.
Then the branch divisor $\mathcal D_\pi$ represents the trivial class in $H_1(\widetilde M, \mathbb Z/2\mathbb Z)$. 

\end{thm}

\section{A central extension of the hyperoctahedral group}

Let $n$ be a positive integer, and let $S_n$ be the symmetric group.
Recall the hyperoctahedral group
\[
B_n = (\mathbb Z/2 \mathbb Z)^n \rtimes S_n
\]
where $S_n$ acts on $(\mathbb Z / 2 \mathbb Z)^n$ by permuting the coordinates.

Let $H_n$ be the group consisting of pairs $(a,b) \in (\mathbb Z/ 2 \mathbb Z)^n \times \mathbb Z/2 \mathbb Z $ with group law 
\[(a_1, b_1) (a_2,b_2) = (a_1+a_2, b_1+ b_2 +  \sum_{1 \leq i< j \leq n }  a_{1,i} a_{2,j}).\] 
A straightforward computation shows that this law is associative, and that the inverse of $(a,b)$ is
\[
(a, b + \sum_{1 \leq i< j \leq n }  a_{i} a_{j}).
\]
Projection onto the first factor exhibits $H_n$ as a central extension of $(\mathbb Z/2 \mathbb Z)^n $ by $\mathbb Z/2 \mathbb Z$.  

For $1 \leq i \leq n$ we donte by $e_i$ the $i$th unit vector in $(\mathbb Z/2 \mathbb Z)^n$,
set $x_i = (e_i, 0) \in H_n$, and $\epsilon = (0,1) \in H_n$.  
We denote the unit element $(0,0) \in H_n$ by $1$.
We can check that 
\begin{equation} \label{OneRelation}
x_i^2 = \epsilon^2 = 1, \quad 1 \leq i \leq n,
\end{equation}  
that
\begin{equation} \label{TwoRelation}
x_i x_j = \epsilon x_j x_i, \quad 1\leq i,j \leq n, \ i \neq j,
\end{equation}
and that
\begin{equation} \label{ThreeRelation}
\epsilon x_i = x_i \epsilon, \quad 1 \leq i \leq n.
\end{equation}

Furthermore, the relations in \cref{OneRelation}, \cref{TwoRelation}, and \cref{ThreeRelation} among the generators $x_1, \dots, x_n, \epsilon$ define the group $H_n$ since using these relations every word in $x_1, \dots, x_n, \epsilon$ can be brought to the form $x_{i_1} \cdots x_{i_k} \epsilon^\delta$ with $1 \leq i_1 < i_2 < \dots < i_k \leq n$ and $\delta \in \{0,1\}$.

We therefore have an action of $S_n$ on $H_n$ by automorphisms via
\[
\sigma (x_i) = x_{\sigma(i)}, \ \sigma(\epsilon) = \epsilon, \quad \sigma \in S_n, \ 1 \leq i \leq n.
\]
Let $G_n =  H_n \rtimes S_n$ be the semidirect product defined using this action. 
Since $\epsilon \in H_n$ is central and $S_n$-invariant, it lies in the center of $G_n$, so 
\[
G_n / \langle \epsilon \rangle = G_n/\{1, \epsilon\} \cong (\mathbb Z/2\mathbb Z)^n \rtimes S_n = B_n. 
\]
We see that $G_n$ is a central extension of $B_n$ by $\mathbb Z/2 \mathbb Z$. 
We denote by $\beta_n$ the class in $H^2( B_n, \mathbb Z/2 \mathbb Z)$ corresponding to this extension.

Let $\sigma, \tau \in B_n$ be two elements that commute, let $\widetilde \sigma, \widetilde \tau$ be lifts to $G_n$, and define
\[
\phi(\sigma, \tau) = [\widetilde \sigma, \widetilde \tau] = \widetilde \sigma \widetilde \tau {\widetilde \sigma}^{-1} {\widetilde \tau}^{-1} \in \langle \epsilon \rangle \cong \mathbb Z/ 2 \mathbb Z.
\]
Since $G_n$ is a central extension of $B_n$, the above is indeed independent of the choice of lifts.
As every element in $\mathbb Z/ 2 \mathbb Z$ is its own inverse, we see that
\begin{equation} \label{SymmetryPhi}
\phi(\sigma, \tau) = [\widetilde \sigma, \widetilde \tau] = [\widetilde \tau, \widetilde \sigma]^{-1} = [\widetilde \tau, \widetilde \sigma] = \phi(\tau, \sigma).
\end{equation}
We denote by
\[
C_{B_n}(\sigma) = \{\tau \in B_n : \sigma \tau = \tau \sigma\}
\]
the centralizer of $\sigma$ in $B_n$.

\begin{prop} \label{HomProp}

For every $\sigma \in B_n$ the map that sends $\tau \in C_{B_n}(\sigma)$ to $\phi(\sigma, \tau)$ is a homomorphism.

\end{prop}

\begin{proof}

For every $\tau_1, \tau_2 \in C_{B_n}(\sigma)$ we have
\[
\phi(\sigma, \tau_1\tau_2) = \widetilde \sigma \widetilde {\tau_1} \widetilde {\tau_2} {\widetilde {\sigma}}^{-1} {\widetilde {\tau_2}}^{-1} {\widetilde {\tau_1}}^{-1}, \quad
\phi(\sigma, \tau_1) \phi(\sigma, \tau_2) = \widetilde \sigma \widetilde {\tau_1} {\widetilde \sigma}^{-1} {\widetilde {\tau_1}}^{-1} [\widetilde \sigma, \widetilde {\tau_2}]
\]
so after cancelling $\widetilde \sigma \widetilde {\tau_1}$, it remains to check that
\[
\widetilde {\tau_2} {\widetilde {\sigma}}^{-1} {\widetilde {\tau_2}}^{-1} {\widetilde {\tau_1}}^{-1} =
{\widetilde \sigma}^{-1} {\widetilde {\tau_1}}^{-1} [\widetilde \sigma, \widetilde {\tau_2}].
\]
After multiplying by $\widetilde \sigma$ from the left, we just need to check that $[\widetilde \sigma, \widetilde {\tau_2}]$ commutes with $\widetilde {\tau_1}$.
This is indeed the case because $[\widetilde \sigma, \widetilde {\tau_2}]$ lies in the central subgroup $\{1, \epsilon\}$ of $G_n$.
\end{proof}

\begin{cor} \label{HomCor}

For every $\tau \in B_n$ the map that sends $\sigma \in C_{B_n}(\tau)$ to $\phi(\sigma, \tau)$ is a homomorphism.

\end{cor}

\begin{proof}

For $\sigma_1, \sigma_2 \in C_{B_n}(\tau)$ we get from \cref{SymmetryPhi} and \cref{HomProp} that
\[
\phi(\sigma_1 \sigma_2, \tau) = \phi(\tau, \sigma_1 \sigma_2) =  \phi(\tau, \sigma_1)  \phi(\tau, \sigma_2) = \phi(\sigma_1, \tau) \phi(\sigma_2, \tau).
\]
as required.
\end{proof}

\begin{prop} \label{PhiForAcycle}

For a $k$-cycle $\sigma = (i_1 \dots i_k) \in S_n \leq B_n$, and 
\[
\tau = e_{i_1} + \dots + e_{i_k} \in (\mathbb Z/ 2 \mathbb Z)^n \leq B_n
\]
we have $\phi(\sigma, \tau) = \epsilon^{k-1}$.
For every $\alpha \in S_n \leq B_n$ with $\alpha(i_1) = i_1, \dots, \alpha(i_k) = i_k$ we have $\phi(\alpha,\tau) = 1$.

\end{prop}

\begin{proof}

We take $\widetilde \sigma = (i_1 \dots i_k)$, $\widetilde \tau = x_{i_1} \cdots x_{i_k}$ and get that
\begin{equation*}
\begin{split}
\phi(\sigma, \tau) &= \widetilde \sigma \widetilde \tau {\widetilde \sigma}^{-1} \cdot {\widetilde \tau}^{-1} = 
\sigma(x_{i_1} \cdots x_{i_k}) \cdot (x_{i_1} \cdots x_{i_k})^{-1} =
x_{\sigma(i_1)} \cdots x_{\sigma(i_k)} \cdot (x_{i_1} \cdots x_{i_k})^{-1} \\
&= x_{i_2} \cdots x_{i_k} x_{i_1} \cdot (x_{i_1} \cdots x_{i_k})^{-1} =
\epsilon^{k-1} x_{i_1} \cdots x_{i_k} \cdot (x_{i_1} \cdots x_{i_k})^{-1}  = \epsilon^{k-1}.
\end{split}
\end{equation*}
Taking $\widetilde \alpha = \alpha$ we see that
\[
\phi(\alpha, \tau) = \widetilde \alpha \widetilde \tau {\widetilde \alpha}^{-1} \cdot {\widetilde \tau}^{-1} = 
\alpha(x_{i_1} \cdots x_{i_k}) \cdot (x_{i_1} \cdots x_{i_k})^{-1} =
x_{\alpha(i_1)} \cdots x_{\alpha(i_k)} \cdot (x_{i_1} \cdots x_{i_k})^{-1} = 1.
\]
as claimed.
\end{proof}

\begin{cor} \label{PhiSigmaComputCor}

Let $\sigma \in S_n \leq B_n$ whose disjoint cycles are 
\[
C_1 = (i_{1,1} \dots i_{1, d_1} ), \dots, C_j  = (i_{j,1} \dots i_{j, d_j}), \quad \sum_{r=1}^j d_r = n,
\]
and let $\tau \in C_{B_n}(\sigma)$. 
Then there exists a (unique) choice of $\tau' \in C_{S_n}(\sigma)$ and $\lambda_1, \dots, \lambda_j \in \mathbb Z / 2 \mathbb Z$
such that 
\begin{equation} \label{ExpressionForTau}
\tau  = \tau' v, \quad v =  \sum_{r=1}^j \lambda_r (e_{i_{r,1}} + \dots + e_{i_{r,d_r}})
\end{equation}
and
\[
\phi(\sigma, \tau) = \epsilon^{\sum_{r=1}^j \lambda_r (d_r - 1)}.
\]

\end{cor}

\begin{proof}

The ability to express $\tau$ as in \cref{ExpressionForTau} is immediate from the definition of the group law in $B_n$.
From \cref{HomProp}, \cref{HomCor}, and \cref{PhiForAcycle} we therefore get that
\begin{equation*}
\begin{split}
\phi(\sigma,\tau) &=  \phi \left(\sigma,  \tau' \cdot \sum_{r=1}^j \lambda_r (e_{i_{r,1}} + \dots + e_{i_{r,d_r}}) \right) = 
\phi(\sigma, \tau') \cdot \prod_{r=1}^j \phi (\sigma, e_{i_{r,1}} + \dots + e_{i_{r,d_r}})^{\lambda_r} \\
&= [\sigma, \tau'] \cdot \prod_{r=1}^j \prod_{s=1}^j
\phi(C_s, e_{i_{r,1}} + \dots + e_{i_{r,d_r}})^{\lambda_r} =
1 \cdot \prod_{r=1}^j \epsilon^{\lambda_r (d_r - 1)} = \epsilon^{\sum_{r=1}^j \lambda_r (d_r - 1)}
\end{split}
\end{equation*}
as required.
\end{proof}

We keep the notation of \cref{PhiSigmaComputCor} and denote by $O_1, \dots, O_z$ the orbits of the action by conjugation of the subgroup of $S_n$ generated by $\tau'$ on $\{C_1, \dots, C_j\}$.
For $1 \leq y \leq z$ we let $I_y \subseteq \{1, \dots, n\}$ be the set of all indices that appear in one of the cycles in $O_y$, and define the permutation $\tau'_y \in S_n$ by 
\[
\tau'_y(i) = 
\begin{cases}
\tau'(i) &i \in I_y \\
i &i \notin I_y.
\end{cases}
\]
We have a disjoint union
\[
\bigcup_{y=1}^z I_y = \{1, \dots, n\}
\]
hence $\tau' = \tau_1' \cdots \tau_z'$ and the permutations $\tau_1', \dots, \tau_z'$ commute.
We put
\[
\tau_y = \tau'_y v_y, \quad v_y = \sum_{\substack{1 \leq r \leq j \\ C_r \in O_y}} \lambda_r (e_{i_{r,1}} + \dots + e_{i_{r,d_r}})
\]
and get that 
\begin{equation} \label{CommutingProdForTauOrL}
\tau = \tau'_1 v_1 \cdots \tau'_z v_z
\end{equation}
where the factors $\tau'_1 v_1, \dots, \tau'_z v_z$ commute.

\section{Proof of \cref{MainRes}}

It suffices to show, for each $\alpha \in H^1(\widetilde M, \mathbb Z/2 \mathbb Z)$, that the pairing of the branch divisor $\mathcal D_\pi$ with $\alpha$ vanishes, namely 
\[\sum_{ \widetilde{K} \ \text{a component of} \ \widetilde{L}} (e_{\widetilde{K}}-1)  \langle[\widetilde{K}], \alpha \rangle = 0\]
 or equivalently
\[
\sum_{K \ \text{a component of} \ L} \ \sum_{ \widetilde{K} \ \text{a component of} \ \pi^{-1}(K)} (e_{\widetilde{K}}-1)  \langle[\widetilde{K}], \alpha \rangle = 0.
\]
Associated to $\alpha$ is a degree two covering space $N \to \widetilde{M}$. 
Let $n$ be the degree of $\pi \colon \widetilde{M} \to M$ which is locally constant away from ${L}$, thus constant. 
Away from $L$, we get that $N$ is a degree $2$ covering space of a degree $n$ covering space, hence has monodromy group contained in the wreath product 
\[
S_2 \wr S_n = (\mathbb Z/2 \mathbb Z) \wr S_n = (\mathbb Z/2 \mathbb Z)^n \rtimes S_n = B_n.
\]  
We thus have a map $H^2( B_n, \mathbb Z/2\mathbb Z) \to H^2 ( M \setminus L, \mathbb Z/2 \mathbb Z)$, and we denote by $\gamma \in H^2 ( M \setminus L, \mathbb Z/2 \mathbb Z)$ the image of $\beta_n$.

Consider a tubular neighborhood $Q $ of $L$ and let $S = \partial Q$ be its boundary, a union of tori. 
Each such torus $T$ corresponds to a unique component $K$ of $L$ - the boundary of a tubular neighborhood of $K$ is $T$.
Since $S$ bounds a 3-manifold in $M \setminus L$, i.e. the complement of the tubular neighborhood $Q$, our cohomology class $\gamma$ integrates to $0$ on $S$. It follows that 
\begin{equation*}
\sum_{T \ \text{a component of} \ S} \int_T \gamma = 0.
\end{equation*}
It is therefore sufficient to prove that 
\begin{equation} \label{IntegralToProve}
\int_T \gamma = \sum_{ \widetilde{K} \ \text{a component of} \ \pi^{-1}(K)} (e_{\widetilde{K}}-1)  \langle[\widetilde{K}], \alpha \rangle.
\end{equation}

Since $T$ is a torus, a covering of $T$ with monodromy $B_n$, i.e. a homomorphism from $\pi_1(T)$ to $B_n$, is given by a pair of elements $m, \ell \in B_n$ that commute, where $m$ represents a meridian and $\ell$ represents a longitude. From the standard cell decomposition of the torus, we can see that 
\begin{equation*}
\int_T \gamma = \phi(m,\ell).
\end{equation*}

Since the $\mathbb Z/2 \mathbb Z$-covering $N \to \widetilde M$ is unbranched over every component $\widetilde K$ of $\pi^{-1}(K)$, 
the monodromy of the meridian $m$ does not swap the two components of the covering, and therefore $m$ is (up to conjugation) contained in $S_n \leq B_n$.




We shall use here the notation of \cref{PhiSigmaComputCor} and the paragraph following it for $\sigma = m$ and $\tau = \ell$, in particular we write $\ell = \ell' v$ as in \cref{ExpressionForTau}.
The components of $\pi^{-1}(K)$ are naturally in bijection with the orbits of the action by conjugation of the subgroup of $S_n$ generated by $\ell'$ on the set of disjoint cycles $\{C_1, \dots, C_j\}$ of $m$.
We denote by $O_{\widetilde K}$ the orbit corresponding to a component $\widetilde K$ of $\pi^{-1}(K)$.
As in \cref{CommutingProdForTauOrL}, we can write 
\begin{equation*} 
\ell = \prod_{\widetilde K} \ell_{\widetilde K}, \quad \ell_{\widetilde K} = \ell'_{\widetilde K} v_{\widetilde K}, \quad v_{\widetilde K} = \sum_{\substack{1 \leq r \leq j \\ C_r \in O_{\widetilde K}}} \lambda_r (e_{i_{r,1}} + \dots + e_{i_{r,d_r}}).
\end{equation*}
We denote the number of cycles in $O_{\widetilde K}$ by $t_{\widetilde K}$, note that each such cycle is of length $e_{\widetilde K}$, and set
\[
d_{\widetilde K} = \# \{1 \leq r \leq j : C_r \in O_{\widetilde K}, \ \lambda_r = 1 \}.
\]

It follows from \cref{PhiSigmaComputCor} that $\phi(m,\ell_{\widetilde K}) \equiv (e_{\widetilde K}-1)d_{\widetilde K} \mod 2$, so from \cref{HomCor} we get that
\[
\phi(m,\ell) = \sum_{\widetilde K \ \text{a component of} \ \pi^{-1}(K)} \phi(m, \ell_{\widetilde K}) \equiv \sum_{\widetilde K \ \text{a component of} \ \pi^{-1}(K)} (e_{\widetilde K}-1)d_{\widetilde K} \mod 2.
\]
It is therefore enough to show that $d_{\widetilde K} \equiv  \langle[\widetilde{K}], \alpha \rangle \mod 2$. 

Let $C$ be a longitude curve in a tubular neighborhood of $\widetilde{K}$. Then $[C] = [\widetilde{K} ]$ as homology classes in $H_1(\widetilde M, \mathbb Z / 2 \mathbb Z)$, so it suffices to show that $d_{\widetilde K} \equiv \langle [C], \alpha \rangle \mod 2$.  The projection of $[C] $ to $T$ is $a[m] + t_{\widetilde K}[\ell]$ for some $a\in \mathbb Z$.  Thus, the action of $C$ on the covering space $N \to M$ is given by $m^a \ell^{t_{\widetilde K}} $. 
We have
\[
m^a \ell^{t_{\widetilde K}}  = m^a (\ell'v)^{t_{\widetilde K}} = m^a {\ell'}^{t_{\widetilde K}} \cdot (v + \ell'(v) + \dots + \ell'^{t_{\widetilde K} - 1}(v) ).
\]



The pairing $\langle [C], \alpha \rangle$ is nonzero if and only if the monodromy along $C$ of the covering $N \to \widetilde M$ is nontrivial, which happens if and only if the action of $m^a \ell^{t_{\widetilde K}}$ sends one branch of this covering to the other, and that occurs if and only if the $k$th entry of $v + \ell'(v) + \dots + \ell'^{t_{\widetilde K} - 1}(v)$ is nonzero for some (equivalently, every) index $1 \leq k \leq n$ that belongs to one of the cycles in $O_{\widetilde K}$. It is therefore sufficient to show that
\[ d_{\widetilde K} \equiv ( v + \ell'(v) + \dots + \ell'^{t_{\widetilde K} - 1}(v) ) _k \mod 2.\]

We have
\[ ( v + \ell'(v) + \dots + \ell'^{t_{\widetilde K} - 1}(v) ) _k = 
v_k + \ell'(v)_k + \dots + \ell'^{ t_{\widetilde K}-1} (v)_k= v_k + v_{ \ell'^{-1} (k) } + \dots +  v_{ \ell'^{ - t_{\widetilde{K}}+1}(k)}.\]
By the orbit-stabilizer theorem, each of the $t_{\widetilde K}$ cycles in $O_{\widetilde K}$ contains exactly one of the $t_{\widetilde K}$ elements $k, \ell'^{-1} (k), \dots, \ell'^{- t_{\widetilde{K}}+1 }(k)$.  
Thus, from \cref{ExpressionForTau} we get that 
\[ v_k + v_{ \ell'^{-1} (k) } + \dots +  v_{ \ell'^{ - t_{\widetilde{K}}+1}(k)}= \sum_{\substack{ 1 \leq r \leq j \\ C_r \in O_{\widetilde K}}} \lambda_r \equiv  d_{\widetilde K} \mod 2,\] as desired.


\end{document}